\documentclass[reqno,11pt]{amsart}
\usepackage{amsmath,amssymb,latexsym,soul,cite,mathrsfs}
\usepackage{color,enumitem,graphicx}
\usepackage[colorlinks=true,urlcolor=blue,
citecolor=red,linkcolor=blue,linktocpage,pdfpagelabels,
bookmarksnumbered,bookmarksopen]{hyperref}
\usepackage[english]{babel}
\usepackage[left=2.6cm,right=2.6cm,top=2.9cm,bottom=2.9cm]{geometry}
\usepackage[hyperpageref]{backref}

\newtheorem{theorem}{Theorem}
\newtheorem{lemma}[theorem]{Lemma}
\newtheorem{corollary}[theorem]{Corollary}
\newtheorem{proposition}[theorem]{Proposition}
\newtheorem{remark}[theorem]{Remark}
\newtheorem{definition}[theorem]{Definition}

\newtheoremstyle{tttheorem}
{}                
{}                
{\slshape}        
{}                
{\bfseries}       
{'}               
{ }               
{}                
\theoremstyle{tttheorem}

\def\XXint#1#2#3{{\setbox0=\hbox{$#1{#2#3}{\int}$ }
		\vcenter{\hbox{$#2#3$ }}\kern-.6\wd0}}

\newcommand{\Ss}{\mathbf{S}}
\newcommand{\R}{\mathbf{R}}

\newcommand{\B}{\mathbf{B}}

\DeclareMathOperator{\Ric}{Ric}

\begin{document}
\title[Moduli of constant $Q$-curvature metrics]{Moduli space theory for complete, 
constant $Q$-curvature metrics on finitely punctured spheres}  

\author[R. Caju]{Rayssa Caju}
\author[J. Ratzkin]{Jesse Ratzkin}
\author[A. Silva Santos]{Almir Silva Santos}

\address[R. Caju]{Department of Mathematical Engineering, University of Chile
\newline\indent 
    Beauchef 851, Edificio Norte, Santiago, Chile}
\email{\href{mailto:rayssacaju@gmail.com}{rcaju@dim.uchile.cl}}

\address[J. Ratzkin]{Department of Mathematics,
	Universit\"{a}t W\"{u}rzburg
	\newline\indent
	97070, W\"{u}rzburg-BA, Germany}
\email{\href{mailto:jesse.ratzkin@uni-wuerzburg.de}{jesse.ratzkin@uni-wuerzburg.de}}

\address[A. Silva Santos]{Department of Mathematics, 
	Federal University of Sergipe
	\newline\indent 
	49107-230, Sao Cristov\~ao-SE, Brazil}
\email{\href{mailto: almir@mat.ufs.br}{ almir@mat.ufs.br}}

\thanks{RC was partially supported by FONDECYT grant number 11230872 and by Centro de Modelamiento Matemático (CMM) BASAL fund FB210005 for center of excellence from ANID-Chile. ASS was partially supported by CNPq grant number 403349/2021-4, 408834/2023-4 and 312027/2023-0}

\maketitle

\begin{abstract}
We study constant $Q$-curvature metrics conformal to the 
the round metric on the sphere with finitely many point singularities. We show that 
the moduli space of solutions with finitely many punctures in fixed 
positions, equipped 
with the Gromov-Hausdorff topology, has the local structure of a real algebraic variety with formal 
dimension equal to the number of the punctures. If a nondegeneracy hypothesis holds, 
we show that a neighborhood in the moduli spaces is actually a smooth, real-analytic 
manifold of the expected dimension. We also construct a geometrically natural 
set of parameters, construct a symplectic 
structure on this parameter space and show that in the smooth case 
a small neighborhood of the moduli space embeds as a Lagrangian 
submanifold in the parameter space. 
\end{abstract}
\section{Introduction} 

In this manuscript we study a fourth-order analog of the singular 
Yamabe problem on finitely punctured spheres, as formulated by Schoen and 
Yau \cite{SY}, Mazzeo, Pollack, and Uhlenbeck \cite{MPU} and others. Our main 
result characterizes the local structure of the moduli space 
of solutions in the case of the round metric on a finitely 
punctured sphere. 

In general, if $(M,g)$ is a Riemannian manifold of dimension 
$n \geq 5$ one defines the $Q$-curvature of $g$ as 
\begin{equation} \label{q_defn}
Q_g = -\frac{1}{2(n-1)} \Delta_g R_g - \frac{2}{(n-2)^2} |\Ric_g|^2 
+ \frac{n^3-4n^2 +16n-16}{8(n-1)^2(n-2)^2} R_g^2,
\end{equation}
where $R_g$ and $\Ric_g$ are the scalar and Ricci curvatures and 
$\Delta_g$ is the Laplace-Beltrami operator. A short computation 
demonstrates that $\displaystyle Q_{\overset{\circ}{g}} = \frac{n(n^2-4)}{8}$, 
where $\overset{\circ}{g}$ is the usual round metric on the sphere 
$\Ss^n$. A longer computation shows that the task of finding a 
conformal metric $\widetilde g = U^{\frac{4}{n-4}} g$ with $Q$-curvature 
equal to $\displaystyle \frac{n(n^2-4)}{8}$ is 
equivalent to solving the nonlinear partial differential 
equation 
\begin{equation} \label{const_q_general}
 \mathcal{H}_g (U) := P_g(U) - \frac{n(n-4)(n^2-4)}{16} 
U^{\frac{n+4}{n-4}}=0
\end{equation} 
where
\begin{equation} \label{paneitz_op_general} 
P_g(U) := (-\Delta_g)^2(U) + \operatorname{div} \left ( \frac{4}{n-2} 
\Ric_g (\nabla U, \cdot )- \frac{(n-2)^2 + 4}{2(n-1)(n-2)} R_g \nabla U\right )
+ \frac{n-4}{2} Q_g U
\end{equation} 
is the Paneitz-Branson operator. This operator is conformally covariant, 
in that 
\begin{equation} \label{trans_law1} 
P_{U^{\frac{4}{n-4}} g} (\phi) 
= U^{-\frac{n+4}{n-4}} P_g (U\phi).
\end{equation} 
Substituting $\phi=1$ into \eqref{trans_law1} we obtain the transformation 
law for $Q$-curvature under a conformal change of metric, which is 
\begin{equation} \label{trans_law2} 
Q_{U^{\frac{4}{n-4}} g} 
= \frac{2}{n-4} U^{-\frac{n+4}{n-4}} P_g(U). 
\end{equation} 

In the case that the background metric is $\overset{\circ}{g}$ the 
Paneitz operator factors as 
\begin{equation}\label{paneitz_op_sph}
P_{\overset{\circ}{g}} = \left ( -\Delta_{\overset{\circ}{g}} + \frac{(n-4)(n+2)}{4} 
\right ) \left ( -\Delta_{\overset{\circ}{g}} + \frac{n(n-2)}{4} \right ).
\end{equation} 
This factorization in some sense simplifies the analysis of \eqref{const_q_general}, 
making the study of the space of solutions in this setting seem 
more approachable. However, the conformal invariance of \eqref{trans_law1}
combined with the noncompactness of the conformal group of 
the round metric leads to the blow-up of sequences of solutions. 
For this reason we study the following singular problem: given a 
closed subset $\Lambda \subset \Ss^n$, describe 
all the conformal metrics $g = U^{\frac{4}{n-4}} \overset{\circ}{g}$ 
that are complete on $\Ss^n \backslash \Lambda$ and have 
$\displaystyle Q_g = \frac{n(n^2-4)}{8}$. We can reformulate 
this geometric problem as the following infinite boundary 
value problem: 
\begin{equation} \label{sing_yamabe_eqn} 
U: \Ss^n \backslash \Lambda \rightarrow (0,\infty), \qquad 
\mathcal{H}_{\overset{\circ}{g}} (U) = 0 , \qquad 
\liminf_{p\rightarrow \Lambda} U(p) = \infty. 
\end{equation} 

We concentrate on the case that $\Lambda = \{ p_1, \dots, p_k\}$ 
is a finite set of points and define the marked moduli space 
$$\mathcal{M}_\Lambda = \left \{ g \in [\overset{\circ}{g}] : Q_g = 
\frac{n(n^2-4)}{8} \mbox{ and } g \textrm{ is complete on }\Lambda \right \}$$
and the unmarked moduli space 
$$\mathcal{M}_k = \left \{ g \in [\overset{\circ}{g}] : Q_g 
= \frac{n(n^2-4)}{8}, \, g \textrm{ is complete on }\Lambda, \, \# \Lambda = k 
\right \} .$$
Here $[\overset{\circ}{g}]$ is the set of all conformal metric to $\overset{\circ}{g}$. We equip both moduli 
spaces with the Gromov-Hausdorff topology. The difference 
between the two is that in the marked moduli space we fix the singular 
points, whereas in the unmarked moduli space we allow them to vary, so long as 
they remain $k$ distinct points.  

Let $\Lambda = \{ p_1, \dots, p_k\} \subset \Ss^n$ and let $g\in 
\mathcal{M}_\Lambda$. Then $g$ admits a definite asymptotic 
structure near each singular point $p_i$, and is asymptotic to one 
of the Delaunay metrics described below in Section \ref{sec:prelim}. 
These Delaunay metrics are, after an appropriate change of variables, 
periodic and uniquely described by their necksizes $\varepsilon \in 
(0,\overline{\varepsilon}]$, where the maximal necksize $\overline{\varepsilon}$ 
depends only on the dimension $n$, which allows one to assign asymptotic 
data to $g \in \mathcal{M}_k$, including the asymptotic necksize 
$\varepsilon_i$ at the singular point $p_i$. We ask the following 
question: how well does this asymptotic data determine a metric 
$g \in \mathcal{M}_\Lambda$? Our results below form a 
first step in answering this question. More precisely, we show that 
under some conditions one can use the asymptotic data to parameterize a 
small neighborhood of moduli space near $g \in \mathcal{M}_\Lambda$. 

Our main theorem is the following result. 
\begin{theorem} \label{loc_struct_thm} 
For each finite subset $\Lambda \subset \Ss^n$ with $\# \Lambda =k\geq 3$ 
the moduli space $\mathcal{M}_\Lambda$ is locally a 
real analytic variety of formal dimension $k$.
\end{theorem}

Our proof follows the road map developed by Kusner, Mazzeo, Pollack 
and Uhlenbeck in \cite{KMP} and \cite{MPU}, combining the implicit function 
theorem and the Lyaponov-Schmidt process. The key technical 
part of our analysis is a fine understanding the linearized operator
$L_g$ of the operator $\mathcal H_g$ defined in \eqref{const_q_general}, which we describe in detail in Section \ref{sec:reformulations}.

As is usually the case, the analysis allows us to make a more precise 
statement if $L_g$ is injective or surjective when acting on an 
appropriate function space. 
\begin{definition} \label{nodegen_defn} 
A metric $g \in \mathcal{M}_\Lambda$ is nondegenerate if $w \in L^2$ and 
$L_g(w) = 0$ implies $w \equiv 0$. 
\end{definition}
\begin{theorem} \label{thm_nondegen_case}
If $g \in \mathcal{M}_\Lambda$ is nondegenerate 
then there exists an open neighborhood $\mathcal{U}\subset \mathcal{M}_\Lambda$
of $g$ that is a real analytic manifold of dimension $k$. 
\end{theorem} 
As a part of proving these two local regularity theorems we construct 
a $2k$-dimensional parameter space $\mathcal{W}_g$ to 
parameterize all metrics $\mathcal{M}_\Lambda$ nearby a given $g$. Once we 
construct $\mathcal{W}_g$ and use it to understand the local regularity of 
the moduli space, we 
construct a symplectic structure on the geometric parameter space and  
show that, in the nondegenerate case, a small neighborhood 
of $g$ in $\mathcal{M}_\Lambda$ embeds in this parameter space 
$\mathcal{W}_g$ as a Lagrangian submanifold.

Our results form a natural progression of the current understanding 
of constant $Q$-curvature metrics. Previously, C. S. Lin \cite{Lin} 
showed that all smooth metrics with constant $Q$-curvature in the 
conformal class of the round metric must be the image of $\overset{\circ}{g}$ 
under a M\"obius transformation. In the language we established above, 
$\mathcal{M}_0 = SO(n+1,1)$, which is the M\"obius group of global conformal 
transformations of the sphere. In the same paper Lin proved there are no 
solutions with a single puncture, {\it i.e.} $\mathcal{M}_1 = \varnothing$ 
and that (after a conformal motion) any solution with two punctures is 
rotationally invariant. Afterwards Frank and K\"onig \cite{FK} characterized 
all the two ended solutions, showing $\mathcal{M}_{p,q} \simeq (0,\overline{\varepsilon}]$
for each $p \neq q$, where $\overline{\varepsilon}$ is a finite, positive 
number depending only on the dimension $n$. We describe these solutions 
in some detail below. In general, explicit gluing constructions demonstrate 
that the moduli space $\mathcal{M}_k$ is 
nonempty, provided $k \geq 2$. Baraket and Rebhi \cite{BR} constructed 
solutions with an even number of punctures by gluing cylinders together, 
using small necks as a bridge. Andrade, Wei and Ye \cite{del-ends2} 
construct many examples in the conformal class of the sphere and 
the authors of this paper together with Andrade and do \`O \cite{del-ends} 
construct many other examples in the inhomogeneous setting, using a different gluing 
technique. Together with Andrade and do \`O \cite{compactness}, the second author 
described a geometric characterization of compact subsets of the moduli space. 

The rest of the paper proceeds as follows. In Section \ref{sec:prelim} 
we discuss some analytic preliminaries, such as the Delaunay solutions, 
the local asymptotics of a singular Yamabe metric near a puncture 
and the appropriate functions spaces. In Section \ref{sec:lin_anal} 
we analyze the mapping properties of the linearized operator $L_g$ 
in various weighted function spaces and introduce the deficiency 
space $\mathcal{W}_g$, a $2k$-dimensional vector space that will serve as a 
parameter space to describe the asymptotic geometry of nearby metrics in 
$\mathcal{M}_\Lambda$. We prove Theorem \ref{thm_nondegen_case}
in Section \ref{sec:nondegen} and complete the proof of Theorem \ref{loc_struct_thm}
in Section \ref{sec:degen}. Finally, in Section \ref{sec:symp} we discuss 
a symplectic structure on the natural parameter space of 
$\mathcal{M}_\Lambda$ and prove that, near smooth points, 
the moduli space $\mathcal{M}_\Lambda$ is a Lagrangian submanifold of 
this parameter space.

\section{Preliminaries} \label{sec:prelim} 

\subsection{The choice of a gauge} \label{sec:reformulations} 

The choice of a gauge in formulating the moduli problem is 
equivalent to choosing the background metric in a conformal 
class. 

While we have thus far phrased this problem in the sphere, it will 
often be useful to rewrite in Euclidean space after stereographic 
projection and to transfer our analysis between the two settings. 
Let $\operatorname{Pr}: \R^n \rightarrow \Ss^n \backslash \{ N \}$ 
be (the inverse of) stereographic projection mapping Euclidean space 
to the sphere minus a pole. It is now a standard exercise to verify 
that 
$$\overset{\circ}{g} = u_{\rm sph}^{\frac{4}{n-4}} \delta, \qquad 
u_{\rm sph} (x) = \left ( \frac{1+|x|^2}{2} \right )^{\frac{4-n}{2}}, $$
where $\delta$ is the Euclidean metric. Using this transformation we can 
identify $g = U^{\frac{4}{n-4}} \overset{\circ}{g}$ with $g= u^{\frac{4}{n-4}} \delta$
where $u = U u_{\rm sph}$. We also denote the preimage of the singular 
set by $\widetilde \Lambda = \operatorname{Pr}^{-1} (\Lambda)$. Without 
loss of generality we let the north pole $N$ be a smooth point of $g$, 
so that the conformal factor $u$ decays at infinity. More precisely, 
$$\limsup_{|x| \rightarrow \infty} |x|^{\frac{n-4}{2}} u(x) < \infty.$$

\begin{remark} 
Hereafter we adopt the convention that capital letters will 
denote conformal factors relative the round metric and lower case 
letters will denote conformal factors relative to the Euclidean 
metric. The two are always related as described above, {\it e.g.} 
$u=U u_{\rm sph}$. 
\end{remark}

Furthermore, we can also rephrase the condition that a metric 
lies in the moduli space $\mathcal{M}_\Lambda$ in the language of 
PDEs. Recall that $g \in [ \overset{\circ}{g}]$ precisely when 
$g = U^{\frac{4}{n-4}} \overset{\circ}{g}$, so that the condition 
$Q_g = \frac{n(n^2-4)}{8}$ becomes 
$$ \mathcal{H}_{\overset{\circ}{g}} (U) = P_{\overset{\circ}{g}}(U) 
- \frac{n(n-4)(n^2-4)}{16} U^{\frac{n+4}{n-4}}=0.$$
Thus 
$$\mathcal{M}_\Lambda = \left \{ U : \Ss^n \backslash \Lambda \rightarrow (0,\infty) : 
\mathcal{H}_{\overset{\circ}{g}} (U) = 0 \mbox{ and } \liminf_{p \rightarrow \Lambda} U(p) 
= \infty \right \}.  $$

In the spherical setting, the linearized operator has the 
form 
\begin{equation} \label{lin_op_spherical} 
L_g (v)  =  \left. \frac{d}{dt} \right |_{t=0} \mathcal{H}_{\overset{\circ}{g}} (U + tv) 
= P_{\overset{\circ}{g}}(v) - \frac{n(n+4)(n^2-4)}{16} U^{\frac{8}{n-4}} v,
\end{equation} 
where the Paneitz operator $P_{\overset{\circ}{g}}$ can be factor as in \eqref{paneitz_op_sph}.

The operators in question have an even simpler appearance in the Euclidean 
setting. This time we use the fact that $g = u^{\frac{4}{n-4}} \delta$, 
so $g \in \mathcal{M}_g$ is now equivalent to the PDE 
$$\mathcal{H}_\delta (u) = (-\Delta_0)^2 u - \frac{n(n-4)(n^2-4)}{16} 
u^{\frac{n+4}{n-4}}=0, $$
which in turn implies 
$$\mathcal{M}_\Lambda = \left \{ u : \R^n \backslash \widetilde \Lambda 
\rightarrow (0,\infty) : \mathcal H_\delta(u)=0, 
\, \liminf_{x \rightarrow \widetilde \Lambda} u(x) = \infty\mbox{ and } 
\limsup_{|x| \rightarrow \infty} |x|^{\frac{n-4}{2}} u(x) < \infty \right \}. $$

In the Euclidean setting the linearized operator has the form 
\begin{equation} \label{lin_op_eucl} 
L_g(v)  =  (-\Delta_0)^2 v - \frac{n(n+4)(n^2-4)}{16} u^{\frac{8}{n-4}} v. 
\end{equation} 

In either setting, we refer to a function satisfying the PDE $L_g(v) = 0$ 
as a Jacobi field. 

\subsection{Delaunay metrics}

The Delaunay metrics are all the constant $Q$-curvature metrics on a 
twice-punctured sphere and, as we will see later, play an important 
role in understanding the behavior of singular constant $Q$-curvature 
metrics with isolated singularities. 

Consider a metric $g = U^{\frac{4}{n-4}} \overset{\circ}{g}$ on
$\Ss^n \backslash \{ p,q\}$ where $p$ and $q$ are distinct. After 
a rotation and a dilation, we can assume $p = N$ is the north pole and 
$q=S$ is the south pole. As in the previous section, we transfer 
now $\R^n \backslash \{ 0 \}$ using stereographic projection and 
let $u = U u_{\rm sph}$. 
Using \eqref{trans_law1} we see that $u: \R^n \backslash \{ 0 \} \rightarrow (0,\infty)$ satisfies 
\begin{equation} \label{del_pde_eucl_coords} 
\mathcal H_\delta(u)=0.
\end{equation} 

Frank and K\"onig \cite{FK} classified 
all the solutions of \eqref{del_pde_eucl_coords}, and we describe them 
here. First we perform the Emden-Fowler change of 
coordinates, defining 
\begin{equation} \label{emden_fowler_coords} 
\mathfrak{F} : \mathcal{C}^\infty (\B_r(0) \backslash \{ 0 \}) 
\rightarrow \mathcal{C}^\infty ((-\log r, \infty) \times \Ss^{n-1}), 
\qquad \mathfrak{F}(u) (t,\theta) = e^{\frac{4-n}{2}t} u(e^{-t}\theta). 
\end{equation} 
We can of course invert $\mathfrak{F}$, obtaining 
$$\mathfrak{F}^{-1}(v) (x) = |x|^{\frac{4-n}{2}} v(-\log |x|, \theta).$$
While the prefactor of $e^{\frac{4-n}{2}t}$ might 
look a little strange at first, a short computation shows it is geometrically 
necessary. Letting 
$$\Upsilon: \R \times \Ss^{n-1} \rightarrow \R^n \backslash 
\{ 0 \} , \qquad \Upsilon (t,\theta) = e^{-t} \theta,$$
we see 
$$\Upsilon^* (\delta)  = e^{-2t} g_{\rm cyl}.$$ 
where $g_{\rm cyl}=dt^2+d\theta^2$ is the cylindrical metric. If we now consider a conformal metric $g= u^{\frac{4}{n-4}} \delta$, we 
see that 
$$\Upsilon^* (g)(t,\theta)  =  \mathfrak{F}(u) (t,\theta)^{\frac{4}{n-4}}
g_{\rm cyl} .$$

After the Emden-Fowler change of coordinates, using \eqref{trans_law1}, \eqref{del_pde_eucl_coords} becomes
\begin{equation} \label{del_pde_cyl_coords} 
\mathcal H_{\rm cyl} (v)=P_{\rm cyl} (v) - \frac{n(n-4)(n^2-4)}{16} v^{\frac{n+4}{n-4}}=0,
\end{equation}
where $v : \R \times \Ss^{n-1} \rightarrow (0,\infty)$ and
\begin{align}
    \label{cyl_paneitz} 
P_{\rm cyl}  & = (-\Delta_{\rm cyl})^2 - \frac{n(n-4)}{2} \Delta_{\rm cyl} 
- 4 \partial_t^2 + \frac{n^2(n-4)^2}{16}\\
& = \partial_t^4 + \Delta_\theta^2+ 2\Delta_\theta \partial_t^2
-\frac{n(n-4)+8}{2}  \partial_t^2 - \frac{n(n-4)}{2} \Delta_\theta + \frac{n^2(n-4)^2}{16}\nonumber
\end{align} is the Paneitz operator of the cylindrical 
metric. Note that $\Delta_{\rm cyl} = 
\partial_t^2 + \Delta_{\Ss^{n-1}}$. C. S. Lin \cite{Lin} used a
moving planes argument to prove that solutions 
of \eqref{del_pde_eucl_coords} are rotationally invariant, 
reducing \eqref{del_pde_cyl_coords} to the ODE 
\begin{equation} \label{del_ode} 
\ddddot v -  \frac{n(n-4)+8}{2} \ddot v + \frac{n^2(n-4)^2}{16} 
v - \frac{n(n-4)(n^2-4)}{16} v^{\frac{n+4}{n-4}}=0. 
\end{equation} 
Notice that one can find a first integral for
this ODE defined as
\begin{equation}\label{eq002}
    \mathcal{H}_\varepsilon = - \dot v_\varepsilon\dddot v_\varepsilon + \frac{1}{2} \ddot v_\varepsilon^2
+  \frac{n(n-4)+8}{4} \dot v_\varepsilon^2- \frac{n^2(n-4)^2}{32} 
v_\varepsilon^2 + \frac{(n-4)^2(n^2-4)}{32} v_\varepsilon^{\frac{2n}{n-4}}.
\end{equation}
We denote the nonzero constant solution of \eqref{del_ode} by 
$$\overline{\varepsilon} = \left ( \frac{n(n-4)}{n^2-4} \right )^{\frac{n-4}{8}}. $$

\begin{theorem} [Frank and K\"onig \cite{FK}] 
For each $\varepsilon \in (0,\overline{\varepsilon}]$ there exists a 
unique $v_\varepsilon: \R \rightarrow (0,\infty)$ solving the 
ODE \eqref{del_ode} attaining its minimal value of $\varepsilon$ 
at $t=0$. All these solutions are periodic. 
Furthermore, let $v:\R \times \Ss^{n-1} \rightarrow (0,\infty)$ be a smooth 
solution of the PDE \eqref{del_pde_cyl_coords}. Then either 
$v(t,\theta) = (\cosh (t+T))^{\frac{4-n}{2}}$ for some $T \in \R$ 
or there exist $\varepsilon \in (0,\overline{\varepsilon}]$ and $T \in \R$
such that $v(t,\theta) = v_\varepsilon (t+T)$. 
\end{theorem}
Later in this paper we will use the fact that the set of Delaunay 
solutions is ordered by the Hamiltonian energy $\mathcal{H}$. In other 
words, $\mathcal{H}$ is a strictly decreasing function of the 
necksize $\varepsilon$. 

We can now write the Delaunay metric in Euclidean coordinates by 
reversing the coordinate transformation \eqref{emden_fowler_coords},
letting  
\begin{equation} \label{del_soln_eucl_coords} 
u_\varepsilon (x) = \mathfrak{F}^{-1}(v)(x) = |x|^{\frac{4-n}{2}} v_\varepsilon (-\log |x|), 
\qquad g_\varepsilon = u_\varepsilon^{\frac{4}{n-4}} \delta 
= v_\varepsilon^{\frac{4}{n-4}} g_{\rm cyl} .
\end{equation} 
The geometric formulation of the Frank-K\"onig classification 
now reads: if $g = U^{\frac{4}{n-4}} \overset{\circ}{g}$ is a constant 
$Q$-curvature metric on $\Ss^n \backslash \{ p,q\}$ then, after a 
global conformal transformation, 
either $g$ extends to smoothly to the round metric or $g$ is 
singular at both $p$ and $q$ and is the image of a Delaunay metric 
$g_\varepsilon$ after said conformal transformation. 

\subsection{Local asymptotics}

A metric $g=U^{\frac{4}{n-4}} \overset{\circ}{g} \in \mathcal{M}_k$ 
with constant $Q$-curvature and finitely many singular points has a 
definite asymptotic structure near each singular point. Let $p_i \in 
\Lambda$ be a singular point of $g$ and choose stereographic coordinates 
$x$ centered at $p_i$. With respect to these coordinates we have 
$g = u^{\frac{4}{n-4}} \delta = (Uu_{\rm sph})^{\frac{4}{n-4}} \delta$ 
there exist $\varepsilon \in (0,\overline{\varepsilon}]$, $R>0$, 
$a \in \R^n$ and $\beta >1$ so that
\begin{equation} \label{asymp1}
u(x) = R^{\frac{n-4}{2}}u_\varepsilon(Rx) + |x|^{\frac{4-n}{2}} \left ( 
\langle x,a \rangle \left ( \frac{n-4}{2} v_\varepsilon (-\log (R|x|))-\dot v_\varepsilon (-\log (R|x|)) 
 \right ) + \mathcal{O}(|x|^\beta) 
\right ) .\end{equation} 
This expansion combines the local asymptotic expansions in \cite{JX} and 
in \cite{jesse2020}. 
As is usually the case, the asymptotic expansion \eqref{asymp1} is more 
tractable in Emden-Fowler coordinates. The transformed function $v=\mathfrak{F}(u)$ 
satisfies the equation \eqref{del_pde_cyl_coords} on the 
half-infinite cylinder $(T_0,\infty) \times \Ss^{n-1}$ and the asymptotic 
expansion now reads 
\begin{equation} \label{asymp2}
v(t,\theta) = v_\varepsilon (t+T) + e^{-t} \langle a,\theta\rangle \left ( 
\frac{n-4}{2} v_\varepsilon(t+T)-\dot v_\varepsilon(t+T)\right ) 
+ \mathcal{O}(e^{-\beta t}), \end{equation}
where $T= -\log R$. 

These asymptotic expansions \eqref{asymp1} and \eqref{asymp2} allow 
us to define an asymptotes map
\begin{equation} \label{defn_asymp_map1}
\mathcal{A}: \mathcal{M}_\Lambda \rightarrow (0,\overline{\varepsilon}]^k 
\times \R^k, \qquad \mathcal{A}(g) = (\varepsilon_1, \dots, \varepsilon_k, T_1,
\dots, T_k), 
\end{equation} 
where $g = u^{\frac{4}{n-4}} \delta$ and 
\begin{equation} \label{defn_asymp_map2} 
u(x) \simeq \mathfrak{F}^{-1} (v_{\varepsilon_i} (-\log |x-p_i| + T_i)) 
\qquad \textrm{ near }p_i.
\end{equation} 
We will see later on, in the proofs of Theorems \ref{loc_struct_thm} and \ref{thm_nondegen_case}, 
that the asymptotes maps provides us with local coordinates for the 
moduli space in the nondegenerate setting. 

\subsection{Weighted function spaces}

We perform most of our analysis below on weighted Sobolev 
spaces. We first define these weighted spaces on a half-infinite 
cylinder, and then transfer the definition to a punctured ball (and thereafter 
to a finitely punctured sphere) using the Emden-Fowler change of coordinates. 
\begin{definition} Let $\delta\in \R$ and let $v\in L^2_{\rm loc} 
((0,\infty) \times \Ss^{n-1})$. We say $v \in L^2_\delta ((0,\infty) 
\times \Ss^{n-1})$ if 
$$\| v\|^2_{L^2_\delta} = \int_0^\infty \int_{\Ss^{n-1}} e^{-2\delta t} 
|v(t,\theta)|^2 d\theta dt < \infty.$$
One can similarly define the Sobolev spaces $W^{k,2}_\delta ((0,\infty) 
\times \Ss^{n-1})$ for any natural number $k$. 
\end{definition} 
Observe that if $|v(t,\theta)| \leq C e^{\widetilde{\delta} t}$ for each 
$\widetilde{\delta} < \delta$ then $v\in L^2_\delta((0,\infty) \times \Ss^{n-1})$. 
Next we undo the Emden-Fowler change of coordinates, letting $u = \mathfrak{F}^{-1}(v)$ 
to see 
\begin{eqnarray*} 
\int_{\Ss^{n-1}} \int_{t_1}^{t_2} e^{-2\delta t} |v(t,\theta)|^2 dt d\theta 
& = & -\int_{\Ss^{n-1}} \int_{e^{-t_1}}^{e^{-t_2}} r^{2\delta+n-5} |u(r\theta)|^2 
dr d\theta \\ 
& = & \int_{r_2 \leq |x| \leq r_1} |x|^{2\delta - 4} |u(x)|^2 d\mu(x),
\end{eqnarray*} 
where $r_1 = e^{-t_1}$ and $r_2 = e^{-t_2}$. Thus we have the following definition. 

\begin{definition} Let $\delta\in \R$, let $r > 0$ and let $u 
\in L^2_{\rm loc} (\mathbf{B}_r(0)\backslash \{ 0 \} )$. We say $u \in 
L^2_\delta (\mathbf{B}_r(0) \backslash \{ 0 \})$ 
if $$\| u \|^2_{L^2_\delta} = \int_{\mathbf{B}_r(0) \backslash \{ 0 \}}
|x|^{2\delta - 4} |u(x)|^2 d\mu < \infty.$$
More generally we let $\widetilde \Lambda \subset \R^n$ be a finite set and 
$u \in L^2_{\rm loc} (\R^n \backslash \widetilde \Lambda)$. We 
say $u \in L^2_\delta(\R^n \backslash \widetilde \Lambda)$ if 
$$\| u\|_{L^2_\delta} = \int_{\R^n \backslash \widetilde \Lambda} 
(\operatorname{dist}(x,\widetilde \Lambda))^{2\delta - 4} |u(x)|^2 d\mu < \infty .$$
\end{definition} 
Once again, we see that if $|u(x)| \leq C(\operatorname{dist}
(x,\widetilde\Lambda))^{2-\widetilde{\delta}}$ for each $\widetilde{\delta} < \delta$ 
then $u \in L^2_\delta(\R^n \backslash \widetilde \Lambda)$.

\section{Linear analysis} \label{sec:lin_anal}

\subsection{The linearization about a Delaunay solution}

Here we study the linearized operator about a Delaunay solution, which we 
denote as $L_\varepsilon$, and 
some of its mapping properties. 

Following Section 5.2 of \cite{del-ends} we write 
$$L_\varepsilon = (-\Delta_0)^2 - \frac{n(n+4)(n^2-4)}{16} 
u_\varepsilon^{\frac{8}{n-4}} $$
and promptly transform to Emden-Fowler coordinates using \eqref{emden_fowler_coords}, 
obtaining the operator $\mathcal{L}_\varepsilon$ defined by 
\begin{equation} \label{del_linearization1} 
\mathcal{L}_\varepsilon (w) (t,\theta) = e^{\frac{4-n}{2} t}
L_\varepsilon (\mathfrak{F}^{-1} (w)) \circ \Upsilon (t,\theta)= 
\mathfrak{F} (L_\varepsilon (\mathfrak{F}^{-1} (w))) (t,\theta). 
\end{equation} 
Some computation gives us 
\begin{eqnarray} \label{del_linearization2} 
\mathcal{L}_\varepsilon & = & P_{\rm cyl} - \frac{n(n+4)(n^2-4)}{16} 
v_\varepsilon^{\frac{8}{n-4}}\nonumber\\
& = & \partial_t^4 + \Delta_{\Ss^{n-1}}^2 + 
2\Delta_{\Ss^{n-1}} \partial_t^2 - \frac{n(n-4)}{2} \Delta_{\Ss^{n-1}} \\ \nonumber 
&&- \frac{n(n-4)+8}{2} \partial_t^2 + \frac{n^2(n-4)^2}{16} - \frac{n(n+4)(n^2-4)}{16} 
v_\varepsilon^{\frac{8}{n-4}}.
\end{eqnarray} 
Here $P_{\rm cyl}$ is given by \eqref{cyl_paneitz}. 

We isolate two specific Jacobi fields of a Delaunay solution: the Jacobi 
field $w_0^+(\varepsilon)$ generating translations along the axis and the Jacobi 
field $w_0^-(\varepsilon)$ generating changes to the necksize. In Emden-Fowler coordinates
these are given by 
\begin{equation} \label{geom_jac_fields}
w_0^+(\varepsilon) = \dot v_{\varepsilon} , \qquad w_0^-(\varepsilon) 
= \frac{d}{d\varepsilon}v_\varepsilon. \end{equation} 
Differentating the relation $v_\varepsilon (t+T_\varepsilon) = v_\varepsilon(t)$
it is straight-forward to verify that $w_0^+$ is bounded and periodic 
while $w_0^-$ grows linearly. The formulation of $w_0^\pm$ above 
is well-formed in the case that $\varepsilon < \overline{\varepsilon}$, but 
both Jacobi fields vanish in the cylindrical case. If $\varepsilon = 
\overline{\varepsilon}$ we define 
\begin{equation} \label{cyl_jac_fields}
w_0^+ = \sin(\sqrt{\mu} t), \qquad w_0^- = \cos(\sqrt{\mu} t), \qquad 
\mu = \frac{\sqrt{n^4-64n+64} - (n^2-4n+8)}{4}. \end{equation} 
The analysis in Proposition 1 of \cite{BR} shows these two 
Jacobi fields play the role of varying the necksize and translation 
parameter on the cylinder.

One can find the following results and their proofs 
in Section 3.6 of \cite{jesse2020}. 

We first write a Jacobi field in Fourier series. Recall that the 
$j$th eigenvalue of $-\Delta_{\Ss^{n-1}}$ is $\lambda_j = j(n-1+j)$ 
and it has multiplicity 
$$m_j = \binom{n-1+j}{j} + \binom{n-3+j}{j-2},$$
and so we can expand $w$ in Fourier series as 
$$w(t,\theta) = \sum_{j=0}^\infty \sum_{l=1}^{m_j} w_{j,l} (t) 
E_{j,l}(\theta),$$
where $\{E_{j,1}, \dots, E_{j, m_j}\}$ is an orthonormal basis of the 
eigenspace of $-\Delta_{\Ss^{n-1}}$ with eigenvalue $\lambda_j$.
Thus the restriction of the operator $\mathcal{L}_\varepsilon$ 
to the eigenspace 
$$\operatorname{Span} \{ E_{j,1}, \dots, E_{j, m_j}\}$$
is the ordinary differential operator 
$$\mathcal{L}_{\varepsilon, j}  =  \frac{d^4}{dt^4} - 
\frac{n(n-4)+8+4\lambda_j }{2}\frac{d^2}{dt^2} + 
\frac{n^2(n-4)^2}{16} + \frac{n(n-4)}{2}\lambda_j +\lambda_j^2 - 
\frac{n(n+4)(n^2-4)}{16} v_\varepsilon^{\frac{8}{n-4}}.$$
\begin{lemma} For each $j \geq 1$ we have $0 \not \in 
\operatorname{spec} (\mathcal{L}_{\varepsilon,j})$
\end{lemma} 

The two functions $w_0^\pm(\varepsilon)$ described above both lie in the 
kernel of $\mathcal{L}_{\varepsilon,0}$, and so $0 \in 
\operatorname{spec}(\mathcal{L}_{\varepsilon,0})$ for each 
$\varepsilon \in (0,\overline{\varepsilon}]$. 

For proof of the next proposition see \cite[Proposition 28]{jesse2020}.

\begin{proposition} \label{prop:del_fredholm} 
There exists a discrete set of real numbers 
\begin{equation} \label{del_indicial_roots} 
\Gamma_\varepsilon = \{ \dots, -\gamma_2(\varepsilon) < 
-\gamma_1(\varepsilon) < 0 < \gamma_1(\varepsilon)  < \gamma_2(\varepsilon), 
\dots \}
\end{equation}
with $\gamma_j(\varepsilon) \rightarrow \infty$ as $j \rightarrow \infty$ 
such that the operator 
$$\mathcal{L}_\varepsilon : W^{4,2}_\delta((0,\infty) \times \Ss^{n-1}) 
\rightarrow L^2_\delta((0,\infty) \times \Ss^{n-1}) $$
is Fredholm provided $\delta \not \in \Gamma_\varepsilon$. In particular, 
for any $\delta \in (0,\gamma_1(\varepsilon))$ 
$$\mathcal{L}_\varepsilon : W^{4,2}_{-\delta} ((0,\infty) \times \Ss^{n-1}) 
\rightarrow L^2_{-\delta} ((0,\infty) \times \Ss^{n-1}) $$
is injective and 
$$\mathcal{L}_\varepsilon : W^{4,2}_\delta ((0,\infty) \times \Ss^{n-1}) 
\rightarrow L^2_\delta ((0,\infty) \times \Ss^{n-1})$$
is surjective. 
\end{proposition} 
One calls $\gamma_j(\varepsilon)$ the $j$th indicial root of the Jacobi 
operator $\mathcal{L}_\varepsilon$ and $\Gamma_\varepsilon$ the set of 
indicial roots associated to the Delaunay solution $v_\varepsilon$. 

\begin{proposition} \label{prop:jac_asymp_expansion} 
Let $\phi\in\mathcal{C}^\infty_0 ((0,\infty) \times \Ss^{n-1})$
and let $\mathcal{L}_\varepsilon(v) = \phi$. Then $v$ satisfies the asymptotic 
expansion $v(t,\theta) \simeq \sum_{j=0}^\infty v_j(t,\theta)$ where each 
$v_j$ is a Jacobi field, {\it i.e.} $L_\varepsilon (v_j) = 0$, and 
$v_j$ decays like a polynomial times $e^{-\gamma_j t}$, where $\gamma_j >0$ 
is the $j$th indicial root. 
\end{proposition} 

\begin{corollary} [Linear Decomposition Lemma I]
Let $\delta \in (0,\gamma_1(\varepsilon))$, let $v \in W^{4,2}_\delta ((0, \infty) \times 
\Ss^{n-1})$ and let $\phi \in \mathcal{C}^\infty((0,\infty) \times 
\Ss^{n-1}) \cap L^2_{-\delta} ((0,\infty) \times \Ss^{n-1})$ be such that 
$\mathcal{L}_\varepsilon (v) = \phi$. Then there exist $z \in W^{4,2}_{-\delta} 
((0, \infty) \times \Ss^{n-1})$ and $w \in \operatorname{Span} (w_0^+(\varepsilon), 
w_0^-(\varepsilon))$  such that $v = z + w$. 
\end{corollary} 
For reasons that will become apparent later in the paper, we 
call $\mathcal{W}_\varepsilon = \operatorname{Span} (w_0^+(\varepsilon), 
w_0^-(\varepsilon))$ the deficiency space associated to the 
Delaunay metric with necksize $\varepsilon$. 

\subsection{The linearization about a singular Yamabe metric}

We transfer the mapping properties of the linearization about a 
Delaunay solution to study the mapping properties of $L_g$, where 
$g \in \mathcal{M}_\Lambda$ is a conformally flat, singular, 
constant $Q$-curvature metric with $k$ prescribed singularities. 
We denote the asymptotic necksize of the puncture $p_j$ by 
$\varepsilon_j$, and define the indicial set 
$$\Gamma_g = \bigcup_{i=1}^k \Gamma_{\varepsilon_i}.$$
It follows directly from Proposition \ref{prop:del_fredholm}
that 
$$L_g : W^{4,2}_\delta (\R^n \backslash \widetilde 
\Lambda) \rightarrow L^2_{\delta-4} (\R^n \backslash 
\widetilde \Lambda)$$
is Fredholm if and only if $\delta \not \in \Gamma_g$. 

\begin{definition} \label{deficiency_defn}
Let $g\in \mathcal{M}_\Lambda$ and choose $r_0>0$ sufficiently 
small such that $\mathbf{B}_{2r_0} (p_i) \cap \mathbf{B}_{2r_0}(p_j) 
= \varnothing$ for each distinct pair of punctures. We define the 
deficiency space $\mathcal{W}_g$ by 
$$\mathcal{W}_g = \operatorname{Span} \{ \chi \mathfrak{F}^{-1} (w_0^+(\varepsilon_i)), 
\chi \mathfrak{F}^{-1} (w_0^-(\varepsilon_i)): i=1, \dots, k\},$$
where $\chi$ is a fixed cut-off function such that 
$$\chi(x) = \left \{ \begin{array}{rl} 1 & |x| < r_0 \\ 0 & |x|> 3r_0/2 
\end{array} \right . , \qquad \|\nabla^k \chi \|_{\mathcal{C}^0} 
\leq c r^{-k}. $$
\end{definition}

\begin{proposition}[Linear decomposition lemma II] 
Let $0 < \delta < \min_{1 \leq i \leq k} \gamma_1(\varepsilon_i)$ and 
let $u\in W^{4,2}_\delta(\R^n \backslash \widetilde \Lambda)$ and 
$\phi \in L^2_{-\delta-4}(\R^n \backslash \widetilde \Lambda)$ satisfying 
$L_g(u) = \phi$. Then there exist $w \in \mathcal{W}_g$ and $v\in 
W^{4,2}_{-\delta}(\R^n \backslash \widetilde \Lambda)$ such that $u=w+v$. 
\end{proposition} 

We now define the bounded null space. Once again we 
fix a number $\delta$ such that $0 < \delta < \min_{1\leq i \leq k} 
\gamma_1(\varepsilon_i)$. Each element of the bounded
null space is, strictly speaking, an equivalence class of functions, 
that is
$$\mathcal{B}_g = \frac{\ker (L_g : W^{4,2}_\delta \rightarrow W^{0,2}_{\delta-4})}
{\ker(L_g: W^{4,2}_{-\delta} \rightarrow W^{0,2}_{-\delta-4})}.$$
Using the Hilbert space structure of $W^{k,2}_\delta$ we can 
identify 
$$\mathcal{B}_g \simeq \{ \ker(L_g: W^{4,2}_\delta\rightarrow 
W^{0,2}_{\delta-4})\} \cap \{ \ker(L_g:W^{4,2}_{-\delta} \rightarrow 
W^{0,2}_{-\delta-4}) \}^\perp.$$
Combining this characterization with the linear decomposition lemma we 
see that one can identify any $v = w+\phi$ for any $v\in \mathcal{B}$, 
where $w \in \mathcal{W}_g$ and $\phi \in W^{4,2}_{-\delta}(\Ss^n 
\backslash \Lambda)$ decays at each puncture. 

Applying Melrose's relative-index calculus we show 
the following dimension count.  
\begin{theorem} 
$\dim(\mathcal{B}_g) = k$
\end{theorem} 

The proof below is more or less the same as the proof of 
Theorem 4.24 in \cite{MPU}. 
\begin{proof} We compute the relative index of $L_g$ acting 
on the appropriate weighted function spaces. Recall that the 
index of 
$$L_g : W^{4,2}_\delta (\R^n \backslash \widetilde \Lambda) 
\rightarrow W^{0,2}_{\delta-4} (\R^n \backslash \widetilde \Lambda)$$
is 
$$\operatorname{ind} (\delta) = \dim (\ker(L_g)) - 
\dim (\operatorname{coker}(L_g)). $$
Integration by parts shows that the $L^2$-adjoint of $L_g$ acting 
on $W^{4,2}_\delta$ is $L_g$ acting on $W^{4,2}_{-\delta}$, and so 
it follows 
$$\operatorname{ind} (-\delta) = - \operatorname{ind}(\delta) $$
provided $\delta \not \in \Gamma_g$. (In this case, reversing the 
sign of the weight $\delta$ exchanges the kernel and the cokernel.) 
Next recall that, provided 
$\delta_1, \delta_2 \not \in \Gamma_g$, the relative index is 
defined as 
$$\operatorname{rel-ind}(\delta_1, \delta_2) = \operatorname{ind}(\delta_1) 
- \operatorname{ind}(\delta_2) .$$
We use duality once more ({\it i.e.} the operator $L_g$ is formally 
self-adjoint in $L^2$) to see so long as $0 < \delta < \min_{1\leq i \leq k} 
\gamma_1(\varepsilon_i)$ we have 
\begin{eqnarray} \label{rel_ind_bnd_null} 
\operatorname{rel-ind} (\delta, -\delta) & = & \operatorname{ind}(\delta) 
- \operatorname{ind} (-\delta) = 2 \operatorname{ind}(\delta) \\ \nonumber
& = & 2 \left ( \dim(\ker(\left. L_g \right |_{W^{4,2}_\delta} ))
- \dim (\ker (\left. L_g \right |_{W^{4,2}_{-\delta}} )) \right ) = 
2 \dim \mathcal{B}_g.
\end{eqnarray}
Thus it suffices to show that $\operatorname{rel-ind} (\delta, -\delta) 
= 2k$ for an appropriate choice of $\delta$. We choose $0 < \delta < 
\min_{1 \leq i \leq k} \gamma_1(\varepsilon_i)$. 

We compute this relative index theorem using the 
Melrose's relative index theorem (see Theorem 6.5 of \cite{Mel}). 
We first decompose 
$$\Ss^n \backslash \{ p_1, \dots, p_k\} = \Omega^c \cup \left ( 
\bigcup_{i=1}^k \mathbf{B}_r(p_i)\backslash \{ p_i\}\right ).$$
We can now write $L_g$ as the sum of restrictions 
$$L_g = \left. L_g \right |_{\Omega^c} + \sum_{i=1}^k  
\left. L_g \right |_{\mathbf{B}_r(p_i) \backslash \{ p_i\}}  $$
and compute the relative index of each restriction separately. It 
might appear that we first have to take the boundary data of 
the restrictions into account, but since 
$$\partial \Omega^c = \bigcup_{i=1}^k \partial \mathbf{B}_r(p_i),$$ 
except for the opposite orientations, the boundary contributions 
in the relative indices will cancel out. Thus it suffices to use the 
Dirichlet boundary data on the $k$ spheres $\{ \partial \mathbf{B}_r(p_i)\}$. 

The operator $L_g$ is elliptic and has index zero on $\Omega^c$, 
so now we're left with computing the relative index of the restriction 
$\left. L_g \right |_{\mathbf{B}_r(p_i) \backslash \{ p_i\} }$ for 
some $i \in \{ 1, \dots, k\}$. Next we observe that, because the relative index is a topological 
invariant, we can deform the metric of $g$ to be {\bf exactly} 
Delaunay in a small neighborhood of each puncture $p_j$. After transforming 
to cylindrical coordinates using the Emden-Fowler change of coordinates $\mathfrak{F}$, 
we finally arrive at the problem of computing the relative index of 
$$\mathcal{L}_{\varepsilon_i} : W^{4,2}_\delta ((0,\infty) \times\Ss^{n-1}) \rightarrow 
W^{0,2}_\delta ((0,\infty) \times \Ss^{n-1}).$$

This is where we use Melrose's machinery, as developed in Chapters 4,5 and 6 
of \cite{Mel}. To do so we introduce the Fourier-Laplace transform 
\begin{equation} \label{Fourier_Laplace_defn}
\widehat{v} (t,\zeta,\theta) = \mathcal{F}_{\varepsilon_i} (v)(t,\zeta,\theta) 
= \sum_{m=-\infty}^\infty e^{-im\zeta} v(t+mT_{\varepsilon_i},\theta)\end{equation}
and the twisted operator 
$$\widetilde{\mathcal{L}_{\varepsilon_i}} : W^{4,2}(\Ss^1_{T_{\varepsilon_i}} \times 
\Ss^{n-1}) \rightarrow W^{0,2}(\Ss^1_{T_{\varepsilon_i}} \times \Ss^{n-1})$$
defined by 
\begin{equation} \label{twisted_op_defn} 
\widetilde{\mathcal{L}_{\varepsilon_i}} (\zeta)(\widehat v) = e^{i\zeta} 
\mathcal{F}_{\varepsilon_i} \circ \mathcal{L}_{\varepsilon_i} \circ 
\mathcal{F}_{\varepsilon_i}^{-1} (e^{-i\zeta} \widehat v).\end{equation}

We make several observations before continuing. First observe that 
$\zeta \in \mathbf{C}$ is a parameter in the Fourier-Laplace transform, 
and the sum in \eqref{Fourier_Laplace_defn} converges precisely when $\zeta$ is in 
the half-space $\{ \zeta : \mathrm{Im} (\zeta) < -\delta T_{\varepsilon_i}\}$. Next 
observe that $\widetilde {\mathcal{L}_{\varepsilon_i}}$ is now a family 
of operators defined between the {\bf fixed} function spaces $W^{4,2}(\Ss^1_{T_{\varepsilon_i}}
\times \Ss^{n-1})$ and $W^{0,2}(\Ss^1_{T_{\varepsilon_i}} \times \Ss^{n-1})$
that depends holomorphically on the complex parameter $\zeta$. This 
allows us to use the analytic Fredholm theorem to conclude that 
$\widetilde{\mathcal{L}_{\varepsilon_i}}$ is Fredholm so long as 
$\zeta$ avoids a discrete set in the complex plan, which in turn 
allows us to define a right-inverse $\widetilde{\mathcal{G}_{\varepsilon_i}}(\zeta)$ 
for $\widetilde{\mathcal{L}_{\varepsilon_i}}$. This right inverse 
$\widetilde{\mathcal{G}_{\varepsilon_i}}$ has a meromorphic extenstion 
to $\mathbf{C}$ with poles at $\widetilde \Gamma_{\varepsilon_i}$. 
In fact, the indicial 
roots $\Gamma_{\varepsilon_i}$ are precisely the imaginary parts of the points 
in $\widetilde \Gamma_{\varepsilon_i}$. 

Melrose's relative index theorem states in this context that the 
relative index is given by a contour integral of the resolvent 
$(\widetilde{\mathcal{L}_{\varepsilon_i}} - \zeta)^{-1}$ about a 
contour surrounding the pole corresponding to the weight $0$, 
as described in the proof of Proposition 26 of \cite{jesse2020} and 
the proof of Proposition 4.15 of \cite{MPU}. This contour integral 
counts the number of tempered, non-decaying Jacobi fields 
with subexponential growth on a Delaunay end. However, we already 
know there are only two such Jacobi fields, namely $w_0^+(\varepsilon_i)$
and $w_0^-(\varepsilon_i)$. We conclude that 
\begin{eqnarray*} 
2 \dim (\mathcal{B}_g) & = & \sum_{i=1}^k \operatorname{ind} 
(\mathcal{L}_{\varepsilon_i}: W^{2,2}_\delta ((0,\infty) \times \Ss^{n-1}) \rightarrow
W^{-2,2}_\delta ((0,\infty) \times \Ss^{n-1}) \\ 
&&  - \sum_{i=1}^k \operatorname{ind} 
(\mathcal{L}_{\varepsilon_i}: W^{2,2}_{-\delta} ((0,\infty) \times \Ss^{n-1}) \rightarrow
W^{-2,2}_{-\delta} ((0,\infty) \times \Ss^{n-1}) \\ 
& = & \sum_{i=1}^k 2 = 2k, 
\end{eqnarray*} 
as we claimed. 
\end{proof} 

\section{Local structure in the nondegenerate case} \label{sec:nondegen}  

In this section we prove local regularity of the moduli space near nondegenerate 
points, as stated in Theorem \ref{thm_nondegen_case}. We first recall the 
statement of the theorem, namely that if $g \in \mathcal{M}_\Lambda$ is 
nondegenerate then there exists an open neighborhood $\mathcal{U}
\subset \mathcal{M}_\Lambda$ of $g$ that is a smooth $k$-dimensional 
manifold. 

\begin{proof} We begin by prescribing the singular set $\Lambda = 
\{ p_1, p_2, \dots, p_k\}$ and choosing a nondegenerate metric 
$g \in \mathcal{M}_\Lambda$. Using the Euclidean gauge, we write $g$ as 
$g = u^{\frac{4}{n-4}} \delta$, where 
$$u:\R^n \backslash \Lambda \rightarrow (0,\infty), \quad 
(-\Delta_0)^2 u = \frac{n(n-4)(n^2-4)}{16} u^{\frac{n+4}{n-4}}, \quad 
\liminf_{x \rightarrow \Lambda} u(x) = \infty.$$
Nondegeneracy of $g$ states that the linearized operator 
$$L_g = (-\Delta_0)^2 - \frac{n(n+4)(n^2-4)}{16} u^{\frac{8}{n-4}}$$
acting on $W^{4,2}(\R^n\backslash \{ p_1, \dots, p_k\})$ has no 
kernel. By the linear decomposition lemma, this is equivalent to the condition that 
\begin{equation} \label{nondeg_ker_cond}
\ker(L_g : W^{4,2}_{-\delta} \oplus \mathcal{W}_g \rightarrow 
W^{0,2}_{-\delta}) = \mathcal{B}_g,\end{equation}
whenever $\delta> 0$ is sufficiently small. 
The bounded null space $\mathcal{B}_g$ always lies in 
the kernel of \eqref{nondeg_ker_cond}, but in the degenerate case 
the kernel will also contain a finite-dimensional space of decaying 
Jacobi fields. 

Intuitively, we would like to describe the metrics in $\mathcal{M}_\Lambda$ 
near $g$ as 
$$\mathcal{U} = \left \{ g_v = (u+v)^{\frac{4}{n-4}} \delta : 
\mathcal{H}_\delta(u+v) = (-\Delta_0)^2(u+v) - \frac{n(n-4)(n^2-4)}{16} 
(u+v)^{\frac{n+4}{n-4}} = 0 \right \}, $$
where $v$ is small with respect to an appropriate norm. If we only allow 
$v$ to decay, the linearized operator does not have any kernel by our 
hypothesis, and so it would be an exercise in futility to construct a 
solution set this way. Furthermore, we should allow the nearby metrics to 
have slightly different asymptotic data, which we cannot encode 
with a decaying perturbing function $v$. On the other hand, if we 
allow perturbing functions $v$ with any order of growth (or even non-decay), 
it is difficult to analyze the zero-set of the operator $\mathcal{H}$, 
and in particular it is impossible to relate the kernel of the 
linearization to this zero-set. We remedy this problem by deforming 
the asymptotic data according to an element of the deficiency space 
$\mathcal{W}_g$, as described below. 

We denote the 
asymptotic necksize of $g$ at the puncture $p_i$ by $\varepsilon_i$. 
Choose $\delta$ such that 
$$0 < \delta < \min_{1 \leq i \leq k} \gamma_1(\varepsilon_i).$$ 
We can identify conformally-related, constant $Q$-curvature metrics in a 
neighborhood of $g$ with 
\begin{equation} \label{deformation_space1}
\mathcal{Z} = \{ (v,w) \in \mathcal{V}_1 \oplus \mathcal{V}_2 \subset 
W^{4,2}_{-\delta} \oplus \mathcal{W}_g : \mathcal{H}(v,w) = 0 \} , 
\end{equation} 
where $\mathcal{V}_1$ and $\mathcal{V}_2$ are small neighborhoods of the origin.  To make sense of this, we should describe the mapping 
\begin{equation} \label{deformation_space2} 
\mathcal{H}: W^{4,2}_{-\delta} \oplus \mathcal{W}_g \rightarrow W^{0,2}_{-\delta}
\end{equation} 
in some detail. By the expansion \eqref{asymp1} (or, equivalently \eqref{asymp2})  
there exist parameters $\varepsilon_i \in (0, \overline{\varepsilon})$, 
$T_i \in \R$ and a decaying function $z \in W^{4,2}_{-\delta} (\B_r(0))$ such that  
$$u(x-p_i) = \mathfrak{F}^{-1} \left ( v_{\varepsilon_i} (-\log |\cdot -p_i| + T_i) 
\right ) + z(x-p_i) . $$

Now let $v \in W^{4,2}_{-\delta}$ and let $w \in \mathcal{W}_g$. By 
definition, 
$$w = \sum_{i=1}^k (a_i^+ w_0^+(\varepsilon_i) + a_i^- w_0^-(\varepsilon_i)),$$
where $a_i^\pm \in \R$. We define the metric $\widetilde g = 
\widetilde u^{\frac{4}{n-4}} \delta$, 
where 
\begin{equation} \label{deformation_space3} 
\widetilde u(x) = \left \{ \begin{array}{cc} u(x) + v(x) & |x-p_i|> 2r_0 \\ 
v(x) + (1-\chi(x)) u(x) + & \\ 
\chi(x)(\mathfrak{F}^{-1}(v_{\varepsilon_i+a_i^-}
(-\log|\cdot - p_i| + T_i + a_i^+)) + z(x-p_i)) & r_0 < |x-p_i| < 2r_0 \\ 
\mathfrak{F}^{-1} (v_{\varepsilon_i + a_i^-} (-\log |\cdot - p_i| + T_i +a_i^+) +z(x-p_i) +v(x-p_i) & 
0 < |x-p_i|< r_0 ,\end{array} \right. \end{equation}
where $r_0$ and $\chi$ are as in Definition \ref{deficiency_defn}. Observe that 
the coefficients $\{ a_i^+, a_i^-\}$ uniquely determine function $w\in \mathcal{W}_g$, 
so the dependence of $\widetilde{u}$ on $w$ is given in how we 
deform the geometric asymptotic data of $g = u^{\frac{4}{n-4}} \delta$. The construction 
of $\widetilde {g}$ is well-defined so long as $\varepsilon_i < \overline{\varepsilon}$, 
but we must adjust it slightly if $\varepsilon_i = \overline{\varepsilon}$. In 
this case we replace 
$$v_{\varepsilon_i + a_i^-} (-\log|\cdot - p_i| + T_i + a_i^+)$$
with 
$$\widetilde{v} (t,\theta) = v_{\overline{\varepsilon}} + 
a_i^- \cos(\sqrt{\mu} (t+T_i + a_i^+) + \mathcal{O}(e^{-t}) ,$$
as constructed in Proposition 1 of \cite{BR}. 

Finally, we identify 
\begin{eqnarray} \label{deformation_op1} 
\mathcal{H}(v,w) & = & \mathcal{H}_g(\widetilde u) = 
\widetilde{u}^{\frac{n+4}{n-4}} \left ( P_{\widetilde g} (1) - 
\frac{n(n-4)(n^2-4)}{16} \right ) \\ \nonumber 
& = & \frac{(n-4)}{2} \widetilde {u}^{\frac{n+4}{n-4}} \left ( 
Q_{\widetilde g} - \frac{n(n^2-4)}{8} \right ). \end{eqnarray}
We also observe that, by construction, $\mathcal{H} (0,0) = 0$.

With this definition, we see that the zero-set $\mathcal{Z}$ is exactly the set of 
constant $Q$-curvature metrics whose asymptotic data are 
close to that of $g$. Observe that we should not expect $\mathcal{H}(0,w)= 0$ 
for any nonzero element of the deficiency space $\mathcal{W}_g$. This is because 
we construct elements of $\mathcal{W}_g$ using a cut-off function $\chi$ 
to transfer deformations of the Delaunay asymptotes to the 
background metric $g$, and so the $Q$-curvature is non-constant 
in the transition region, where $\nabla \chi \neq 0$. However, the 
quantity $\displaystyle Q_{\widetilde g} - \frac{n(n^2-4)}{8}$ is small 
(assuming $w$ is small) and compactly supported. Thus we expect to 
be able correct the $Q$-curvature with a decaying function $v \in W^{4,2}_{-\delta}$, 
exactly as described above. Additionally, the linearization of the 
operator $\mathcal{H}_g$ as applied to $W^{4,2}_{-\delta} \oplus 
\mathcal{W}_g$ is 
$$L_g : W^{4,2}_{-\delta} \oplus \mathcal{W}_g \rightarrow 
W^{0,2}_{-\delta-4} $$
and
\begin{equation} \label{ker_linearization}
\ker \left (L_g: W^{4,2}_{-\delta} \oplus \mathcal{W}_g \rightarrow 
W^{0,2}_{-\delta-4}\right ) = \mathcal{B}_g \oplus \ker \left (L_g: W^{4,2}_{-\delta} \rightarrow 
W^{0,2}_{-\delta-4}\right ).\end{equation} 
However, since $g$ is nondegenerate the second summand on the right 
hand side of \eqref{ker_linearization} is just to $0$ function. 
Thus the kernel of $L_g$ is precisely $\mathcal{B}_g$, 
which has dimension $k$, which is also its minimal possible dimension. 
Thus $\dim(\ker(L_{\widetilde g})) = k$ on an open neighborhood of 
$g$ in $\mathcal{M}_\Lambda$, and so by the implicit function theorem 
an open neighborhood of $\mathcal{Z}$ containing $g$ is a smooth, 
$k$-dimensional manifold. 
\end{proof} 

\section{Local structure in the degenerate case} \label{sec:degen} 

Our purpose is to discuss the local structure of the moduli space $\mathcal{M}_{\Lambda}$ 
without the hypothesis that the linearization has a trivial $L^{2}$-nullspace. In this 
context, we apply the Lyapunov-Schmidt argument as presented in \cite{KMP}. The key idea 
goes back to Simon's proof of an infinite-dimensional version of the \L ojasiewicz
inequality, see Theorem 3 of \cite{Si}.

\begin{theorem} The space $\mathcal{M}_{\Lambda}$ is locally a finite dimensional real analytic variety.
\end{theorem}

\begin{proof}
Once again, our problem can be reduced to the understanding of the zero set of 
$$\mathcal{H} : W^{4,2}_{-\delta} \oplus \mathcal{W}_g \rightarrow 
W^{0,2}_{-\delta}, $$
where $\mathcal{H}$ is defined in  \eqref{deformation_space3} and \eqref{deformation_op1}. 
This time, however, the kernel of the linearization that we shall denote as $K \equiv\ker \left (L_g: W^{4,2}_{-\delta} \rightarrow 
W^{0,2}_{-\delta-4}\right)$ is nontrivial in $ \subset W^{4,2}_{-\delta}$, and it can identify by duality with the cokernel of $L_g : W^{4,2}_\delta \rightarrow W^{0,2}_{\delta-4}.$

Following \cite{KMP} we define 
$$\widetilde{\mathcal{H}}: W^{4,2}_{-\delta} \oplus \mathcal{W}_g\oplus K 
\rightarrow W^{0,2}_{-\delta}, \qquad \widetilde{\mathcal{H}}(v,w,\phi) 
= \mathcal{H}(v,w) + \phi ,$$
so that 
\begin{align*}
\mathcal{Z} & = \{ (v,w) \in \mathcal{V}_1 \oplus \mathcal{V}_2 : 
\mathcal{H}(v,w) = 0 \} \\ & = \{ (v,w,\phi) \in \mathcal{V}_1 \oplus \mathcal{V}_2
\oplus \mathcal{V}_3 : \widetilde{\mathcal{H}} (v,w,\phi) = \phi \}.  
\end{align*}
where $\mathcal{V}_1\subset W^{4,2}_{-\delta}$, $\mathcal{V}_2 \subset 
\mathcal{W}_g$ and $\mathcal{V}_3\subset K$ are small neighborhoods of the origin in each respective Banach space. 

We now see that 
\begin{eqnarray} 
\mathcal{Z} & \subset & \widetilde{\mathcal{Z}} = \{ (v,w,\phi) \in  \mathcal{V}_1 \oplus \mathcal{V}_2
\oplus \mathcal{V}_3 : 
\widetilde{\mathcal{H}}(v,w,\phi) \in K_{(1)} \} \\ \nonumber 
& = & \{ (v,w,\phi) \in  \mathcal{V}_1 \oplus \mathcal{V}_2
\oplus \mathcal{V}_3: \Pi^\perp (\widetilde{\mathcal{H}}(v,w,\phi)) = 0 \} 
= \ker (\Pi^\perp \circ \widetilde{\mathcal{H}}), 
\end{eqnarray} 
where $\Pi^\perp$ is the orthogonal projection of $W^{0,2}_{-\delta}$
onto $K^\perp$. The linearization of this operator is given by
$$\Pi^\perp \circ L_g: K^\perp \oplus \mathcal{W}_g \oplus K \rightarrow 
W^{0,2}_{-\delta},$$
which is now a surjective operator. Furthermore, we can characterize the 
kernel of the linearization as 
$$\ker (\Pi^\perp \circ L_g) = \{ (v,w,\phi) \in W^{4,2}_{-\delta} \oplus \mathcal{W}_g 
\oplus K : L_g (v+w) \in K\} \simeq K \oplus \mathcal{B}_g .$$
Thus by the implicit function theorem, there is a 
real-analytic function 
$$\Psi : K \oplus \mathcal{B}_{g} \oplus K \rightarrow (W^{4,2}_{-\delta} 
\oplus \mathcal{W}_g) /(K \oplus \mathcal{B}_g), \qquad \Psi(v,w,\phi) = 
(\psi_1(v,w,\phi), \psi_2(v,w,\phi)) $$
such that 
$$\widetilde{\mathcal{Z}} = \{ (\psi_1(v,w,\phi),\psi_2(v,w,\phi),\phi): 
(v,w,\phi) \in K \oplus \mathcal{B}_g \oplus K \} .$$

Unraveling these definitions we see 
$$\mathcal{Z} \simeq \{ (v,w,\phi) \in K \oplus \mathcal{B}_g \oplus K : 
\mathcal{H} (v+\psi_1(v,w,\phi), w+\psi_2(v,w,\phi)) = 0 \},$$
which is indeed the zero set of an analytic function acting on a 
finite-dimensional vector space. This proves a small neighborhood 
$\mathcal{U} \subset \mathcal{M}_\Lambda$ containing $g$ is indeed 
a real-analytic variety. 
\end{proof}

\section{Symplectic structure} \label{sec:symp}

Here we discuss the asymptotes mapping from  
the marked moduli space $\mathcal{M}_\Lambda$ into a fixed 
configuration space $\mathbb{M}_\Lambda = (0,\overline{\varepsilon}]^k 
\times \R^k$, where each pair $(\varepsilon_i, T_i)$ characterizes
the Delaunay asymptote at the puncture $p_i$. We further 
show that if $g \in \mathcal{M}_\Lambda$ is nondegenerate then 
this local mapping is a Lagrangian embedding with respect to 
the standard symplectic structure. One can 
construct a similar asymptotes map for the unmarked moduli 
space, and much of the properties we prove below carry through, 
but in this latter case, the configuration spaces are 
larger and constructing the mapping is more involved. 

First we construct a symplectic form on $\mathcal{M}_\Lambda$. Let 
$g = U^{\frac{4}{n-4}} \overset{\circ}{g} \in \mathcal{M}_\Lambda$ 
and transfer $g$ to $\R^n \backslash \widetilde \Lambda$ using stereographic projection, rewriting $g = u^{\frac{4}{n-4}} \delta$ with 
$u = U u_{\rm sph}$. For any sufficiently small $r>0$ we define
$$\Omega_r = \R^n \backslash \left ( \bigcup_{i=1}^k \B_r(p_i) \right )$$
and 
\begin{equation} \label{defn_symp_form}
\omega(v,w) = \lim_{r \searrow 0} \int_{\Omega_r} (v L_g (w) - 
w L_g (v) )d\mu_0.\end{equation} 
Here $d\mu_0$ is the Euclidean volume element, $v,w \in \mathcal{W}_g$ 
lie in the deficiency space of $g$ (See Definition \ref{deficiency_defn}) and $L_g$ is the Jacobi operator of $g$, 
which is defined in \eqref{lin_op_eucl}.

\begin{theorem} \label{thm:symp_form} 
The form $\omega$ defined in \eqref{defn_symp_form} is a 
symplectic form on the $2k$-dimensional vector space $\mathcal{W}_g$. 
\end{theorem}

\begin{proof} Our first order of business is to show that 
$\omega$ is well-defined, {\it i.e.} that the limit in \eqref{defn_symp_form}
exists. By \eqref{lin_op_eucl} observe that 
$$vL_g(w) - wL_g(v) = v \Delta_0^2 w - w \Delta_0^2 v.$$
Next, we recall that the outer unit normal of $\Omega_r$ is 
$-\partial_r$ on each boundary sphere $\partial \B_r(p_i)$ 
and integrate by parts to see 
$$\int_{\Omega_r} v\Delta_0^2 w d\mu_0 = \int_{\Omega_r} 
(\Delta_0 v) (\Delta_0 w) d\mu_0 - \sum_{i=1}^k \int_{\partial 
\B_r(p_i)} (v \partial_r \Delta_0 w - \partial_r v \Delta_0 w
)d\sigma_0,$$
and so \eqref{defn_symp_form} becomes 
 \begin{equation} \label{symp_form2} 
 \omega(v,w) = \lim_{r \searrow 0} \sum_{i=1}^k 
 \int_{\partial \B_r(p_i)} (w \partial_r \Delta_0 v - v \partial_r \Delta_0 w
 + \partial_r v \Delta_0 w - \partial_r w \Delta_0 v )d\sigma_0.\end{equation} 

 Next, we change variables using \eqref{emden_fowler_coords}, letting 
 $\widetilde v = \mathfrak{F}(v)$ and $\widetilde w= \mathfrak{F}(w)$. Under this 
 change of variables 
 \begin{equation} \label{transformed_derivatives1} 
 \partial_r v = -e^{ \frac{n-2}{2}t} 
 \left ( \partial_t \widetilde v + \frac{n-4}{2} \widetilde v \right ), 
 \qquad \Delta_0 v = e^{\frac{n}{2}t} \left ( \partial_t^2 \widetilde v - 2 \partial_t 
 \widetilde v -\frac{n(n-4)}{4} \widetilde v + \Delta_\theta \widetilde v \right )
 \end{equation} 
 and 
 \begin{equation} \label{transformed_derivatives2} 
 \partial_r \Delta_0 v = -e^{\frac{n+2}{2}t} \left ( \partial_t^3 
 \widetilde v + \frac{n-4}{2} \partial_t^2 \widetilde v -\frac{n^2}{4} \partial_t \widetilde v
 -\frac{n^2(n-4)}{8} \widetilde v + \Delta_\theta \partial_t \widetilde v + \frac{n}{2} \Delta_\theta 
 \widetilde v \right ).\end{equation}

 Plugging \eqref{transformed_derivatives1} and \eqref{transformed_derivatives2}
 into \eqref{symp_form2} we obtain 
 \begin{align*}
   \int_{\partial \B_r(p_i)} & \left(w \partial_r \Delta_0 v - v \partial_r \Delta_0 w + 
 \partial_r v \Delta_0 w - \partial_r w \Delta_0 v\right) d\sigma_0\\  
 = & \int_{\Ss^{n-1}} (w \partial_r \Delta_0 v - v \Delta_0 w + \partial_r v \Delta_0 w - \partial_r w  \Delta_0 v )(r\theta)r^{n-1} d\theta\\
 = &\int_{\Ss^{n-1}} \left ( -\widetilde{w}\partial_t^3 \widetilde{v} - \frac{n-4}{2}  \widetilde{w}
 \partial_t^2 \widetilde{v} + \frac{n^2}{4} \widetilde{w}\partial_t \widetilde{v} + \frac{n^2(n-4)}{8} \widetilde{w}\widetilde{v} 
 - \widetilde{w}\Delta_\theta \partial_t \widetilde{v} - \frac{n}{2} \widetilde{w}\Delta_\theta \widetilde{v}  \right.\\
 & +\widetilde{v}\partial_t^3 \widetilde{w} + \frac{n-4}{2}  \widetilde{v}
 \partial_t^2 \widetilde{w} - \frac{n^2}{4} \widetilde{v}\partial_t \widetilde{w} - \frac{n^2(n-4)}{8} \widetilde{v}\widetilde{w} 
 + \widetilde{v}\Delta_\theta \partial_t \widetilde{w} + \frac{n}{2} \widetilde{v}\Delta_\theta \widetilde{w} \\ 
 &-\partial_t \widetilde{v} \partial_t^2 \widetilde{w} + 2 \partial_t \widetilde{v} 
 \partial_t \widetilde{w} + \frac{n(n-4)}{4} \widetilde{w} \partial_t \widetilde{v} - \partial_t \widetilde{v} 
 \Delta_\theta \widetilde{w}  - \frac{n-4}{2} \widetilde{v} \partial_t^2\widetilde{w} + (n-4)\widetilde{v} 
 \partial_t\widetilde{w}\\
 &+ \frac{n(n-4)^2}{8} \widetilde{v}\widetilde{w} - \frac{n-4}{2} \widetilde{v} \Delta_\theta
 \widetilde{w}\\ 
 &+\partial_t \widetilde{w} \partial_t^2 \widetilde{v} - 2 \partial_t \widetilde{w} 
 \partial_t \widetilde{v} - \frac{n(n-4)}{4} \widetilde{v} \partial_t \widetilde{w} + \partial_t \widetilde{w} 
 \Delta_\theta \widetilde{v}  + \frac{n-4}{2} \widetilde{w} \partial_t^2\widetilde{v} - (n-4)\widetilde{w} 
 \partial_t\widetilde{v}\\
 &\left.- \frac{n(n-4)^2}{8} \widetilde{w}\widetilde{v} + \frac{n-4}{2} \widetilde{w} \Delta_\theta
 \widetilde{v} \right)d\theta\\
 = & \int_{\Ss^{n-1}} \left ( \widetilde{v} \partial_t^3 \widetilde{w} - \widetilde{w} \partial_t^3 \widetilde{v} 
 + \partial_t\widetilde{w} \partial_t^2 \widetilde{v} - \partial_t{v} \partial_t^2\widetilde{w} 
 +\frac{n(n-4)+8}{2} (\widetilde{w}\partial_t \widetilde{v} - \widetilde{v} \partial_t\widetilde{w} 
 )\right ) d\theta
 \end{align*}

 Observe that each element of the deficiency space is asymptotically radial about 
 each puncture point, so that the expansions 
 $$\mathfrak{F} (v(\cdot - p_i)) (t,\theta) = \overline{v}(t) + \mathcal{O} (e^{-\delta t}), 
 \qquad \mathfrak{F}(w(\cdot - p_i)) (t,\theta) = \overline{w}(t) + \mathcal{O}(e^{-\delta t})$$
 for each puncture $p_i$. Thus each term in the expansion above involving derivatives with 
 respect to $\theta$ will vanish in the limit. 

 Next we use the fact that both $v$ and $w$ lie in the definciency space 
 $\mathcal{W}_g$ (see Definition \ref{deficiency_defn}). This means that near each puncture $p_i$ the 
 functions $v$ and $w$ have asymptotic expansions of the form 
 \begin{equation} \label{expansion_def_space}
 \widetilde v = \alpha_i^+ w_0^+(\varepsilon_i) + \alpha_i^- w_0^-(\varepsilon_i), \qquad 
 \widetilde w = \beta_i^+ w_0^+(\varepsilon_i) + \beta_i^- w_0^-(\varepsilon_i)  
 \end{equation} 
 for $t$ sufficiently large.

 At each end, using bilinearity and skew-symmetry, we see that $\omega(v,w)$ is the limit as $t\rightarrow\infty$ of
\begin{eqnarray}\label{eq001}
    (\alpha_i^+\beta_i^--\alpha_i^-\beta_i^+)\int_{\Ss^{n-1}}\left(w_0^+ \dddot {w}_0^- 
     -  w_0^- \dddot {w}_0^+ + \dot {w}_0^- \ddot {w}_0^+ - \dot {w}_0^+ \ddot {w}_0^-
     \right .&& \\ \nonumber 
     \qquad \qquad \qquad \qquad \qquad  \left .  + \frac{n(n-4)+8}{2} (w_0^- \dot {w}_0^+ - w_0^+ \dot {w}_0^-)\right)d\theta & & 
\end{eqnarray}

If $A_\varepsilon$ is the integrand in \eqref{eq001}, then
\begin{align*}
\frac{d}{dt}A_\varepsilon
& = w_0^+ \left(\ddddot w_0^--\frac{n(n-4)+8}{2}\ddot w_0^-\right) - w_0^- \left(\ddddot w_0^+-\frac{n(n-4)+8}{2}\ddot w_0^+\right)\\
& =\left ( 
\frac{n^2(n-4)^2}{16} - \frac{n(n+4)(n^2-4)}{16} v_\varepsilon^{\frac{8}{n-4}} \right )\left(w_0^+w_0^--w_0^-w_0^+\right)=0.
\end{align*}

Thus $A_\varepsilon$ does not depend on $t$. Here we have used the ODE for $w_0^\pm$, namely 
$$\ddddot w_0^\pm - \frac{n(n-4)+8}{2} \ddot w_0^\pm + \left ( 
\frac{n^2(n-4)^2}{16} - \frac{n(n+4)(n^2-4)}{16} v_\varepsilon^{\frac{8}{n-4}} \right ) w_0^\pm 
= 0.$$

Let us find the value of the integrand in \eqref{eq001} at $t=0$ using the definitions of $w_0^\pm$ in \eqref{geom_jac_fields}. First observe that $w_0^+(\varepsilon) = \dot v_\varepsilon$, 
so $w_0^+ (0)  =  \dot v_\varepsilon (0) = 0$, $\dot w_0^+ (0)  =  \ddot v_\varepsilon (0) > 0$, $\ddot w_0^+(0)  =  \dddot v_\varepsilon (0) = 0$ and by \eqref{del_ode} we get
$$\dddot w_0^+ (0) = \ddddot v_\varepsilon (0)  =  \frac{n(n-4)+8}{2} \ddot v_\varepsilon(0) - \frac{n(n-4)}{16} 
\left ( n(n-4) \varepsilon - (n^2-4)\varepsilon^{\frac{n+4}{n-4}} \right ) .
$$
Furthermore, since $\displaystyle w_0^- = \frac{d}{d\varepsilon} v_\varepsilon$
and $v_\varepsilon$ assumes its minimal value at $t=0$ we see 
$w_0^-(0) =1$. Therefore, at $t=0$ it holds
\begin{align*}
    A_\varepsilon(0) & =-\dddot w_0^+(0)-\ddot v_\varepsilon(0)\ddot w_0^-+\frac{n(n-4)+8}{2}\ddot v_0^+(0)\\
    & =-\ddot v_\varepsilon(0)\ddot w_0^-+\frac{n(n-4)}{16} 
\left ( n(n-4) \varepsilon - (n^2-4)\varepsilon^{\frac{n+4}{n-4}} \right )
\end{align*}
By \eqref{eq002} we have
\begin{equation} \label{eq003}
\mathcal{H}_\varepsilon = \frac{1}{2} \ddot v_\varepsilon (0)^2 - \frac{n^2(n-4)^2}{32} 
\varepsilon^2 + \frac{(n-4)^2(n^2-4)}{32} \varepsilon^{\frac{2n}{n-4}} . 
\end{equation} 
As we remarked earlier, the energy is a decreasing function of $\varepsilon$, minimized by the 
cylinder, which has the largest possible necksize, so differentiating \eqref{eq003} we see 
\begin{align}
    0  >  \frac{d}{d\varepsilon} \mathcal{H}_\varepsilon & = \ddot v_\varepsilon (0) 
\frac{d}{d\varepsilon} \ddot v_\varepsilon (0) - \frac{n^2(n-4)^2}{16} \varepsilon 
+ \frac{n(n-4)(n^2-4)}{16} \varepsilon^{\frac{n+4}{n-4}} \\ \nonumber 
 & =  \ddot v_\varepsilon (0) \ddot w_0^- (0) - \frac{n(n-4)}{16}(n(n-4) \varepsilon 
- (n^2-4)\varepsilon^{\frac{n+4}{n-4}} ) .
\end{align}
This implies that
$$A_\varepsilon(0)=-\frac{d}{d\varepsilon}\mathcal H_\varepsilon>0,$$
and so $\omega$ is nondegenerate.
 \end{proof}

 \begin{corollary} 
 Let $g \in \mathcal{M}_\Lambda$ be nondegenerate. Then there exists an 
 open neighborhood $\mathcal{U}$ of $g$ in $\mathcal{M}_\Lambda$ that 
 embeds into $\mathcal{W}_g$ as a Lagrangian submanifold, with respect to 
 the symplectic form given by \eqref{defn_symp_form}. 
 \end{corollary} 

 \begin{proof} 
 As in the proof of Theorem \ref{thm_nondegen_case}, we can 
 identify the bounded null space $\mathcal{B}_g$ as the tangent space 
 $T_g\mathcal{M}_\Lambda$. In particular, this identification shows 
 $L_g(v) = 0$ for each $v \in \mathcal{B}_g$. On the other hand, the 
 linear decomposition lemma allows us to identify $\mathcal{B}_g$
 as a $k$-dimensional subspace of $\mathcal{W}_g$. The corollary now follows.  
 \end{proof} 

\bibliography{moduli_references}
\bibliographystyle{amsplain}

\end{document}